\documentclass[12pt,letterpaper,reqno]{amsart}
\usepackage{bm}
\usepackage{mathrsfs}
\usepackage{latexsym}
\usepackage{epsfig}
\usepackage{color}
\usepackage{comment}
\usepackage{float}
\usepackage{tikz,tikz-cd}
\usepackage{amsmath,amsfonts,amsthm,amssymb,amscd,hyperref}
\usepackage[all,dvips,arc,curve,frame]{xy}

\newcommand{\supp}{\operatorname{supp}}
\newcommand{\A}{\mathcal{A}}
\newcommand{\Z}{\mathbb{Z}}
\newcommand{\B}{\mathscr{B}}

\newcommand{\N}{\mathbb{N}}
\newcommand{\T}{\overline{T}}

\begin{document}

\newtheorem{thm}{Theorem}[section]
\newtheorem{conj}[thm]{Conjecture}
\newtheorem{cor}[thm]{Corollary}
\newtheorem{lem}[thm]{Lemma}
\newtheorem{prop}[thm]{Proposition}
\newtheorem{exa}[thm]{Example}
\newtheorem{defi}[thm]{Definition}
\newtheorem{rem}[thm]{Remark}

\numberwithin{equation}{section}

\title[Intrinsic ergodicity of the squarefree flow]{Uniqueness of the measure of maximal entropy for the squarefree flow} 
\author{Ryan Peckner}

\maketitle

\begin{abstract}
The squarefree flow is a natural dynamical system whose topological and ergodic properties are closely linked to the behavior of squarefree numbers. We prove that the squarefree flow carries a unique measure of maximal entropy and express this measure explicitly in terms of a skew-product of a Kronecker and a Bernoulli system. Using this characterization and a number-theoretic argument, we then show that the unique maximum entropy measure fails to possess the Gibbs property.    
\end{abstract}

\section{Introduction}\label{sec: intro}
One of the most important objects in analytic number theory is the M\"obius function, defined for positive integers $n$ by
\[ \mu(n) = \left\{
\begin{array}{cl}
1 &\ \text{if }n=1 \\
0 &\ \text{if }n\text{ is not squarefree}\\
(-1)^{r} &\ \text{if }n=p_{1}\cdots p_{r}\text{ is a product of $r$ distinct primes}.
\end{array}\right. 
\]
It is well-known that if the sequence $\left(\mu(n)\right)_{n\geq 1}$ behaves randomly in the sense that
\[
\sum_{n\leq N}\mu(n) = O_{\epsilon}\left(N^{1/2 + \epsilon}\right) \ \text{ for any }\epsilon>0,
\]
then the Riemann hypothesis is true. Less quantitative but richer reflections of the chaotic behavior of $\mu$ are captured by instances of the ``M\"obius Randomness Law'', see for example \cite{IK}.

Sarnak has recently developed an approach to the idea of M\"obius randomness based on the study of certain dynamical systems \cite{Sar}. Let $\Omega^{(3)} = \{-1,0,1\}^{\mathbb{N}}$ with the product topology, and let $\sigma:\Omega^{(3)}\to\Omega^{(3)}$ be the left shift defined by $(\sigma x)_{n} = x_{n+1}$. $(\Omega^{(3)},\sigma)$ is a topological dynamical system, that is, a pair $(X,S)$ where $X$ is a compact metric space and $S: X\to X$ is continuous and surjective. 

The M\"obius function defines a point $\mu = \left(\mu(1), \mu(2), \mu(3),\dots\right)\in\Omega^{(3)}$, and we let $\mathcal{M}$ be the closure in $\Omega^{(3)}$ of the orbit $\{\sigma^{k}\mu : k\geq 0\}$. Then the `M\"obius flow' $\mathcal{M}$ is a closed, $\sigma$-invariant subset of $\Omega^{(3)}$, but due to the mysterious behavior of $\mu$ one cannot say much else about it. The purpose of this paper is to study a related system about which we can say a great deal. 

Let $\Omega^{(2)} = \{0,1\}^{\mathbb{N}}$ and let $\phi: \Omega^{(3)}\to\Omega^{(2)}$ be the squaring map $(x_{n})\mapsto\left(x_{n}^{2}\right)$. Observe that $\phi$ is continuous, surjective, and intertwines the shift maps on the two spaces (this is an example of a \emph{factor map} between topological dynamical systems). Let $\mathcal{S}$ be the closure in $\Omega^{(2)}$ of the orbit of $\phi(\mu) = \left(\mu^{2}(n)\right)_{n\geq1}$ under $\sigma$. Since $\mu^{2}(n)$ is the indicator function of the set of squarefree numbers, we call $(\mathcal{S},\sigma)$ the squarefree flow.

Naturally, the dynamical structure of $\mathcal{S}$ is strongly tied to the statistical properties of squarefree numbers. In this connection, Sarnak uses the well-known squarefree sieve (cf. \cite{Tsa}) to produce a certain ergodic measure $\nu$ of zero entropy on $\Omega^{(2)}$ whose support is $\mathcal{S}$ (for the definitions of terms from ergodic theory see \cite{Wal}). This allows him in turn to prove the following: define the \emph{support} of a sequence $x = (x_{n})\in\Omega^{(2)}$ to be the set $\supp(x)\subset\mathbb{N}$ of those $n$ such that $x_{n}=1$. We call a subset $A\subset\mathbb{N}$ \emph{admissible} if for every prime $p$ the reduction $A\bmod{p^{2}}$ is a proper subset of $\mathbb{Z}/p^{2}\mathbb{Z}$. We then call a sequence $(x_{n})\in\Omega^{(2)}$ admissible if its support is an admissible subset of $\mathbb{N}$. As, for every prime $p$, the support of $(\mu^{2}(n))_{n\geq 1}$ is missing the residue class $0\bmod{p^{2}}$, we see that $\mu^{2}$ is admissible in this sense.

Let $\mathcal{A}\subset\Omega^{(2)}$ be the set of all admissible sequences. It's easy to show that $\mathcal{A}$ is a closed, shift invariant subset of $\Omega^{(2)}$. Using the aforementioned measure $\nu$ (and in particular the fact that its support is $\mathcal{S}$), Sarnak is able to show that in fact $\mathcal{A} = \mathcal{S}$. Though this may seem surprising at first, it arises quite naturally from the shape of the measure $\nu$, which in turn is constructed from the main term of the squarefree sieve formula. 

The ergodicity of $\nu$ implies that the subset $\mathcal{A}_{1}$ of $\mathcal{A}$ consisting of sequences whose support omits only one residue class $\mod{p^{2}}$ for every $p$ has full $\nu$-measure. It's then not hard to show that the map which sends such a sequence to the point in $\mathcal{K} := \prod_{p}\mathbb{Z}/p^{2}\mathbb{Z}$ whose $p$-th coordinate is the residue class omitted by the sequence $\bmod{p^{2}}$ intertwines the shift on $\mathcal{A}$ with the translation $T_{(-1,-1,-1,\dots)}$ on $\mathcal{K}$ by the element $(-1,-1,-1,\dots)$ (note that this map is Borel measurable but not continuous). Moreover, the pushforward of $\nu$ under this map is precisely the mass one Haar measure $m$ on the compact group $\mathcal{K}$, so we have a factor map of \emph{measure-preserving} dynamical systems $(\mathcal{A}_{1},\sigma,\nu)\to(\mathcal{K},T_{(-1,-1,-1,\dots)},m)$. Subsequently, Cellarosi and Sinai \cite{CS} used spectral theory techniques to prove that this map is in fact an isomorphism of 
measure-preserving dynamical systems (it is 
however very far from being an isomorphism of \emph{topological} dynamical systems). 

While this tells the whole story for the measure $\nu$, there are other invariant measures on $\mathcal{A}$ that should be taken into account. Specifically, as stated by Sarnak (and proven in this paper), the system $(\mathcal{A},\sigma)$ has topological entropy $(6/\pi^{2})\log2$ (it's no coincidence that $6/\pi^{2}$ is the density of the squarefree numbers in $\mathbb{N}$!). The variational principle (e.g. \cite{Wal} 8.2) states that the topological entropy of a topological dynamical system is the supremum of the measure entropies over all its invariant probability measures, and it is known that any subshift system possesses at least one \emph{measure of maximal entropy}, i.e. an invariant probability measure whose measure entropy equals the topological entropy of the system. Thus there exists such a measure for $\mathcal{A}$. 

While (ergodic) measures of maximal entropy always exist for irreducible subshifts (this is not the case for all topological dynamical systems), the number of such measures may be arbitrarily large, even infinite (see \cite{Hay} for examples). In keeping with the general principle in ergodic theory that a scarcity of measures is more meaningful than an abundance of them, it is an important problem to determine when there is only one measure of maximal entropy on a given topological dynamical system. 

The aim of this paper is to prove the following.
\begin{thm}
The squarefree flow $\mathcal{A}$ possesses a unique measure of maximal entropy. 
\label{thm: main}
\end{thm}

The essence of our proof is the construction of a Borel measurable map (eq. (\ref{eq: insertion}))
\[
 \iota: \left(\prod_{p}\Z/p^{2}\Z\right)\times\Omega^{(2)}\to \mathcal{A}
\]
such that $\iota\circ T = \sigma\circ\iota$, where $T$ is a certain skew-product transformation on $\left(\prod_{p}\Z/p^{2}\Z\right)\times\Omega^{(2)}$. This map codifies the idea of constructing a sequence with admissible support by inserting zeros into an arbitrary sequence in $\Omega^{(2)}$ along the residue classes described by an element of $\prod_{p}\Z/p^{2}\Z$. 

Once this factor map has been constructed, the problem effectively reduces to proving the intrinsic ergodicity of the skew-product transformation $T$. Measures of maximal entropy for precisely such skew-products have been studied in great detail in \cite{MarNew}. As their results are much more general than the case treated here, we have opted to give a direct explication of their proof in the case of interest to us, where the base of the skew product is uniquely ergodic with entropy zero. The key ideas of this argument - namely, inducing the skew product to a subset where it decomposes as a direct product, then using Abramov's formula to compare entropies - are entirely borrowed from their work. Our contribution lies in the use of the factor map $\iota$ above to compare the complicated system $\mathcal{A}$ with a more transparent skew-product system.

This construction yields an explicit formula for the effect of the unique measure of maximal entropy on the cylinder set corresponding to any word in $\A$, and by using this formula we are able to show in Section \ref{sec: gibbs} that it is not a Gibbs measure, in contrast to many familiar classes of intrinsically ergodic systems. 

We note that this method also applies to the $\mathscr{B}$-free shifts studied in \cite{ELR}, which are defined as follows. Fix an infinite set $\B = \{b_{1},b_{2},b_{3},\dots\}\subset\N^{*}$ satisfying the following properties:
 \begin{enumerate}
  \item $\forall 1\leq r < r'$, $b_{r}$ and $b_{r'}$ are relatively prime
  \item $\sum_{r\geq1}b_{r}^{-1} < \infty$.
 \end{enumerate}
 
 Integers with no factors in $\B$ are called $\B$-free. We define the characteristic function
 \[
  \nu(n) = \begin{cases}
             0 & \mbox{ if } n \mbox{ is not $\B$-free}\\
             1 & \mbox{ if } n \mbox{ is $\B$-free}. \\
            \end{cases}
 \]
 
 The $\B$-free flow $X$ is then the closure of the point $(\nu(n))_{n\geq1}\in\Omega^{(2)} = \{0,1\}^{\N^{*}}$ under the left shift, and can be identified as the set of sequences the reduction of whose support modulo each $b_{r}$ is a proper subset of $\Z/b_{r}\Z$ (\cite{ELR} Cor. 4.2). Thus the squarefree flow is the $\mathscr{B}$-free flow for $\mathscr{B} = \{p^{2} : p\text{ prime} \}$, and as no special properties of squares of primes are used in our argument for intrinsic ergodicity aside from their mutual relative primality, the proof carries over to the general case essentially verbatim. In contrast, our argument regarding the Gibbs property is quite specific to squares of primes, and it is unclear whether this may be extended to general sets $\mathcal{B}$. 
 
 While this work was under review, the preprint \cite{KLW} appeared, in which a detailed study of invariant measures for $\B$-free systems is presented, including a proof of our main Thm. \ref{thm: main} for all such systems. Their approach, while substantially different from ours, is also based on relating an invariant measure on the squarefree flow to the Haar measure on a compact monothetic group and the $(1/2,1/2)$ Bernoulli measure. 
 
 I wish to thank my advisor Peter Sarnak for many helpful conversations, as well as Manfred Einsiedler for pointing out the relevance 
of skew products in connection with the insertion map mentioned above. I am also indebted to an anonymous referee for many detailed suggestions and for pointing out several errors in earlier versions of this paper. 

\section{Induced transformations}
\label{sec:induced}

We first recall some generalities on induced transformations, as they are essential to our argument. Suppose $T: X\to X$ is a Borel isomorphism of a compact metric space, and $Y\subset X$ is a recurrent Borel subset, meaning $Y\subseteq \cup_{n=1}^{\infty}T^{-n}Y$. Let $r : Y\to \mathbb{N}$ be the return-time map, so $r(x) = \min\{n\geq 1 : T^{n}x\in Y\}$, and define the induced transformation $T_{Y}: Y\to Y$ by $T_{Y}x = T^{r(x)}x$. Given a $T$-invariant Borel probability measure $\mu$ on $X$ such that $\mu(Y) > 0$, we may define a $T_{Y}$-invariant measure $\mu_{Y}$ on $Y$ by $\mu_{Y}(E) = \mu(E\cap Y)/\mu(Y)$. We then have the following well-known facts (cf. \cite{Zwei},\cite{Th}). 
\begin{enumerate}
  \item Suppose $\mu$ is a $T$-invariant Borel measure on $X$ such that $X = \cup_{n=0}^{\infty}T^{-n}Y$ up to a set of $\mu$-measure zero. Then $\mu_{Y}$ is ergodic for $T_{Y}$ if and only if $\mu$ is ergodic for $T$. 
  \item We may calculate the entropy of $\mu_{Y}$ as $h_{\mu_{Y}}(T_{Y})\mu(Y) = h_{\mu}(T)$.
  \item Let $\nu$ be a $T_{Y}$-invariant Borel probability measure on $Y$. Then there exists a $T$-invariant Borel probability measure $\mu$ on $X$, concentrated on $\cup_{n=0}^{\infty}T^{-n}Y$, such that $\nu = \mu_{Y}$.  
\end{enumerate}
Fact (2) is Abramov's formula \cite{Ab}. Facts (1) and (3) may be easily proven: (1) follows because for any $n\geq 1$, $\{x\in Y : r(x) = n\}\cap T^{-n}(E\cap Y) = \{r(x) = n\}\cap T^{-n}E$, so that $E\cap Y$ is invariant for $T_{Y}$ whenever $E$ is invariant for $T$. For (3), the measure $\mu$ can be given explicitly as
\[
 \mu(E) = \sum_{n\geq 0}\nu(\{r(x) > n\}\cap T^{-n}E).
\]
The $T$-invariance of $\mu$ follows by writing
\[
 \mu(T^{-1}E) = \sum_{n\geq 1}\nu(\{r(x) > n\}\cap T^{-n}E) + \sum_{n\geq 1}\nu(\{r(x) = n\}\cap T^{-n}E)
\]
and using the $T_{Y}$-invariance of $\nu$. It follows that the assignment $\mu\mapsto \mu_{Y}$ is a bijection between $T$-invariant measures on $X$ with $\mu(Y) > 0$ and $T_{Y}$-invariant measures on $Y$. 

\section{The Squarefree Flow}

We now turn to our situation. We will be dealing with Borel measurable maps that are not continuous, and so must be somewhat careful in our definitions. Following \cite{MarNew}, if $S$ is a Borel measurable isomorphism of a compact metric space, we define the topological entropy
\begin{equation}
 h(S) = \sup\{h_{\mu}(S) : \mu\mbox{ is an $S$-invariant Borel probability measure}\}.
\label{eq: entropy}
\end{equation}
When $S$ is a homeomorphism, this agrees with the usual definition of $h_{\text{top}}(S)$ by the variational principle. If the $\sup$ is achieved by a unique measure, $S$ is called intrinsically ergodic. 

Define $G = \prod_{p}\Z/p^{2}\Z$, and define the translation function $T_{1} : G\to G$ by $T_{1}(g) = (g_{1} -1,g_{2} - 1, \dots)$. Let $G_{1} = \{g\in G : g_{p}\equiv 1\bmod{p^{2}}\mbox{ for some } p\}$, and define the transformation $T: G\times \Omega^{(2)}\to G\times \Omega^{(2)}$ by 
\[
 T(g,y) = \begin{cases}
           (T_{1}g,y) & \mbox{ if } g\in G_{1} \\
           (T_{1}g,\sigma y) & \mbox{ if not},\\
          \end{cases}
\]
where $\sigma$ is the left shift on $\Omega^{(2)}$. $T$ is a Borel isomorphism of $G\times\Omega^{(2)}$, but is not continuous since $G_{1}$ is open and not closed in $G$. Observe that we can describe $T$ as a skew-product transformation: define the Borel measurable function $\psi = 1 - \chi_{G_{1}} : G\to\{0,1\}$, where $\chi_{G_{1}}$ is the indicator function of $G_{1}$. Then 
\begin{equation}
T(g,y) = (T_{1}g, \sigma^{\psi(g)}y).
\label{eq: skew prod}
\end{equation}
Let $X$ be the squarefree flow as defined above. The key to our argument is an ``insertion map'' $\iota: G\times\Omega^{(2)}\to X$ that allows us to realize $X$ in terms of a more transparent system. Let's describe the motivation for this construction. A sequence $x\in\Omega^{(2)}$ belongs to $X$ if and only if $x$ omits at least one arithmetic progression modulo $p^{2}$ for all primes $p$, i.e. if and only if there exists $a\in\{1,2,\dots,p^{2}\}$ such that $x_{a +\ell p^{2}} = 0$ for any $\ell\geq0$. Outside of these arithmetic progressions, however, the entries of $x$ may be chosen with complete freedom. Thus, suppose we begin with an element $g\in G$ and \emph{any} sequence $x'\in\Omega^{(2)}$. To define a sequence in $X$ from this information, we start reading the entries of $x'$ until we reach a position congruent to one of the $g_{p}\bmod{p^{2}}$, where we insert a zero. We then continue to read $x'$ from where we left off until we reach another position congruent to some (likely to be different) $g_{
q}\bmod{q^{2}}$, 
where we again insert a zero, and so on. For example, let's carry out this process with respect to only one prime $p = 2$. Suppose we are given the residue class $3\pmod{p^{2}=4}$ and the sequence $x'\in\Omega^{(2)}$ shown below. Then modifying $x'$ as just described will yield the indicated sequence $x$:
\begin{align*}
x' & = 01110110101101\cdots \\
x & = 01\underbar{0}110\underbar{0}110\underbar{0}101\underbar{0}101\underbar{0}\cdots
\end{align*}
where the underlined zeros have been inserted to force $x$ to omit the progression $3\bmod{4}$. To formalize this, observe for instance that the fifth entry of $x$ equals the fourth entry of $x'$, since there is exactly one $n\leq 5$ with $n\equiv 3\bmod{4}$, and $5 - 1= 4$. At the same time, the eigth entry of $x$ equals the sixth entry of $x'$, since now there are two $n\leq 8$ with $n\equiv 3\bmod{4}$, and $8 - 2 = 6$.

Thus, for each $g\in G$ define the function $\alpha_{g}:\N\to\N$ by 
\[
\alpha_{g}(m) = |\{1\leq n\leq m : n\equiv g_{p}\bmod{p^{2}}\text{ for some } p\}|.
\]
Now, for $(g,y)\in G\times\Omega^{(2)}$, let $\iota(g,y)\in X$ be the sequence defined by
\begin{equation}
 \iota(g,y)_{m} = \begin{cases}
                       0 & \mbox{ if }m\equiv g_{p}\bmod{p^{2}}\mbox{ for some }p \\
                       y_{m-\alpha_{g}(m)} & \mbox{ if not}. \\
                      \end{cases}
\label{eq: insertion}
\end{equation}

$\iota$ is clearly surjective, but not continuous, since one must read the entire infinite sequence $g\in G$ before deciding whether to insert a zero at a given position $m$. 

\begin{prop}
 $\iota$ is Borel measurable, and we have $\iota\circ T = \sigma\circ\iota$.  
\end{prop}

We first will prove the equivariance property, then demonstrate the measurability. To prove the desired relation, we must divide into cases dictated by $G_{1}$. Thus, suppose $g\in G_{1}$, and let $y\in\Omega^{(2)}$. Then we have for any $m\geq 1$
\[
 \iota(T(g,y))_{m} = \iota(T_{1}g,y)_{m} = \begin{cases}
                                            0 & \mbox{ if }m + 1\equiv g_{p}\bmod{p^{2}}\mbox{ for some } p \\
                       y_{m-\alpha_{T_{1}g}(m)} & \mbox{ if not}. \\
                      \end{cases}                                            
\]
Observe that if $\not\exists p$ with $m + 1\equiv g_{p}\bmod{p^{2}}$, then since $g\in G_{1}$, we have
\begin{align*}
 \alpha_{T_{1}g}(m) & = |\{1\leq n\leq m : n + 1\equiv g_{p}\bmod{p^{2}} \text{ for some }p\}| \\
  & = |\{1\leq n\leq m : n\equiv g_{p}\bmod{p^{2}} \text{ for some } p\}\backslash\{1\}| \\
    & = \alpha_{g}(m) - 1.
\end{align*}
Therefore, for such $m$ we have 
\[
 y_{m-\alpha_{T_{1}g}(m)} = y_{m - \alpha_{g}(m) + 1}.
\]
At the same time, 
\[
 \sigma(\iota(g,y))_{m} = \iota(g,y)_{m+1} = \begin{cases}
                                              0 & \mbox{ if }m + 1\equiv g_{p}\bmod{p^{2}}\mbox{ for some } p \\
                                              y_{(m+1) - \alpha_{g}(m+1)} & \mbox{ if not}. \\
                                             \end{cases}
\]
In the second case above, we clearly have $\alpha_{g}(m + 1) = \alpha_{g}(m)$; therefore $y_{(m+1) - \alpha_{g}(m + 1)} = y_{m - \alpha_{g}(m) + 1} = \iota(T(g,y))_{m}$. In the first case, it's obvious that $\iota(T(g,y))_{m} = \sigma(\iota(g,y))_{m} = 0$; therefore we have equality for all $m\geq 1$ and so $\iota(T(g,y)) = \sigma(\iota(g,y))$.

Now assume $g\not\in G_{1}$. Then 
\[
\iota(T(g,y))_{m} = \iota(T_{1}g,\sigma y)_{m} = \begin{cases}
                                            0 & \mbox{ if }m + 1\equiv g_{p}\bmod{p^{2}}\mbox{ for some } p \\
                       y_{m-\alpha_{T_{1}g}(m) + 1} & \mbox{ if not}. \\
                      \end{cases}
\]
If $m\geq 1$ and $\not\exists p$ with $m + 1\equiv g_{p}\bmod{p^{2}}$, we have 
\begin{align*}
 \alpha_{T_{1}g}(m) & = |\{1\leq n\leq m : n + 1\equiv g_{p}\bmod{p^{2}} \text{ for some }p\}| \\
  & = |\{1\leq n\leq m : n\equiv g_{p}\bmod{p^{2}} \text{ for some } p\}| \\
    & = \alpha_{g}(m).
\end{align*}
Thus in this case, $y_{m-\alpha_{T_{1}g}(m) + 1} = y_{m - \alpha_{g}(m) + 1}$. On the other hand,
\[
 \sigma(\iota(g,y))_{m} = \iota(g,y)_{m+1} = \begin{cases}
                                              0 & \mbox{ if }m + 1\equiv g_{p}\bmod{p^{2}}\mbox{ for some } p \\
                                              y_{(m+1) - \alpha_{g}(m+1)} & \mbox{ if not}. \\
                                             \end{cases} 
\]
In the second case above, we clearly have $\alpha_{g}(m + 1) = \alpha_{g}(m)$; therefore $y_{(m+1) - \alpha_{g}(m + 1)} = y_{m - \alpha_{g}(m) + 1} = \iota(T(g,y))_{m}$. In the first case, it's obvious that $\iota(T(g,y))_{m} = \sigma(\iota(g,y))_{m} = 0$; therefore we have equality for all $m\geq 1$ and so $\iota(T(g,y)) = \sigma(\iota(g,y))$. So this holds both for $g\in G_{1}$ and $g\not\in G_{1}$, proving the desired equivariance.

To prove the measurability of $\iota$, for each $r\geq1$ define  
\[
 X(r) = \{x\in \Omega^{(2)} : \supp(x) \mbox{ omits at least one residue class }\bmod{p_{i}^{2}}\mbox{ for } i=1,\cdots, r\}.
\]
Note that $X(r)\supset X$ for all $r\geq 1$. For each $r\geq 1$, let $G_{r} = \prod_{i=1}^{r}\Z/p_{i}^{2}\Z$, and let $\pi_{r}: G\times \Omega^{(2)} \to G_{r}\times \Omega^{(2)}$ be the product of projection $G\to G_{r}$ onto the first $r$ coordinates with the identity on $\Omega^{(2)}$. We define the partial insertion map $\phi_{r}: G_{r}\times\Omega^{(2)}\to X(r)$ as follows. For each $g\in G_{r}$ define the function $\alpha_{g}^{r}:\N\to\N$ by 
\[
\alpha_{g}^{r}(m) = |\{n\leq m : n\equiv g_{p_{i}}\bmod{p_{i}^{2}}\text{ for some } i\in\{1,\dots,r\}\}|.
\]
Now, for $(g,y)\in G_{r}\times\Omega^{(2)}$, let $\phi_{r}(g,y)\in X(r)$ be the sequence defined by
\begin{equation}
 \phi_{r}(g,y)_{m} = \begin{cases}
                       0 & \mbox{ if }m\equiv g_{p_{i}}\bmod{p_{i}^{2}}\mbox{ for some }i\in\{1,\dots,r\} \\
                       y_{m-\alpha_{g}^{r}(m)} & \mbox{ if not}. \\
                      \end{cases}
\label{eq: partial}
\end{equation}

$\phi_{r}$ is a surjective map from $G_{r}\times\Omega^{(2)}$ onto $X(r)$. Define $\iota_{r} : G\times\Omega^{(2)}\to X(r)$ to be the composition of $\pi_{r}$ with the partial insertion map $\phi_{r}: G_{r}\times\Omega^{(2)}\to X(r)$; then $\iota_{r}$ is surjective as well. 

\begin{prop}
 $\iota_{r}$ is continuous for each $r\geq 1$, and we have $\iota_{r}\to\iota$ pointwise everywhere on $G\times\Omega^{(2)}$. Therefore, $\iota$ is Borel measurable.  
 \label{prop: pointwise}
\end{prop}

\begin{proof}
For a subshift $\Sigma\subset\Omega^{(2)}$ and a word $w$ of length $\ell$ appearing in $\Sigma$ (meaning there exists some $x\in \Sigma$ such that $w = x_{1}\cdots x_{\ell}$), we will denote the cylinder set corresponding to $w$ by
 \[
 C_{w} = \{x\in \Sigma: x_{1}\cdots x_{\ell} = w\}.
 \]
Such sets provide a base for the topology on $\Sigma$.

To show that $\iota_{r}$ is continuous, let $(g,y)\in G\times\Omega^{(2)}$ and let $M \geq 1$. Let $u = y_{1}\cdots y_{M}$ be the first length $M$ subword of $y$, and consider the subset
\[
 U = \left(\{g_{1}\}\times\{g_{2}\}\times\cdots\times\{g_{r}\}\times\prod_{k > r}\Z/p_{k}^{2}\Z\right)\times C_{u}\subset G\times\Omega^{(2)}.
\]
$U$ is clearly an open subset of $G\times\Omega^{(2)}$, and it's equally clear that the sequences $\iota_{r}(g,y)$ and $\iota_{r}(h,z)$ agree to at least $M$ positions for any $(h,z)\in U$. As $M$ was arbitrary, this proves that $\iota_{r}$ is continuous. 

 We now prove the pointwise everywhere convergence $\iota_{r}\to\iota$. Let $(g,y)\in G\times\Omega^{(2)}$, and fix a word $w$ of length $\ell$ appearing in $X$. We will show that $\iota(g,y)\in C_{w}$ if and only if $\iota_{r}(g,y)\in C_{w}$ for all sufficiently large $r$ (where `sufficiently large' is taken relative to $(g,y)$ since we only seek pointwise convergence). Observe that there exists an $R\geq1$ such that 
\[
 \{m\leq\ell : m\equiv g_{r}\bmod{p_{r}^{2}}\text{ for some }r\} = \{m\leq\ell : m\equiv g_{r}\bmod{p_{r}^{2}}\text{ for some }r\leq R\}.
\]
Indeed, we have 
\[
 \{r : g_{r}\equiv m\bmod{p_{r}^{2}}\text{ for some }m\leq\ell\} = \coprod_{m\leq\ell}\{r : g_{r}\equiv m\bmod{p_{r}^{2}}\}
\]
so we may choose an $r$ from each nonempty set on the right and take $R$ to be the maximum of these. Then for any $n\leq\ell$ we have 
\begin{align*}
 \alpha_{g}(n) & = \alpha_{g}^{R}(n) \\
               & = |\{m\leq n : m\equiv g_{r}\bmod{p_{r}^{2}}\text{ for some }r\leq R\}|.  
\end{align*}
Thus if $S\geq R$ we have for any $n\leq\ell$
\begin{align*}
 \iota(g,y)_{n} & = \begin{cases}
                       0 & \text{ if } n\equiv g_{r}\bmod{p_{r}^{2}}\text{ for some } r \\
                        y_{n - \alpha_{g}(n)} & \text{ if not}
                     \end{cases} \\
              & = \begin{cases}
                    0 & \text{ if } n\equiv g_{r}\bmod{p_{r}^{2}}\text{ for some } r\leq R \\
                    y_{n -  \alpha_{g}^{R}(n)} & \text{ if not}
                  \end{cases} \\
              & = \iota_{S}(g,y)_{n},
\end{align*}
and it follows that $\iota_{S}(g,y)\in C_{w}$ for $S\geq R$ if and only $\iota(g,y)\in C_{w}$, proving the claim.
\end{proof}
For any infinite sequence of integers $I = (i_p)\in\prod_{p}\{1,\dots,p^{2}-1\}$ (so $i_{p}$ is the number of residue classes omitted mod $p^{2}$) consider the set 
\[
\{x\in X: \supp(x)\mbox{ omits exactly } i_{p}\mbox{ residue classes }\bmod{p^{2}} \ \forall p\}.
\]
This set isn't invariant, since there may exist a sequence $x$ in this set and a residue class $h\bmod{p^{2}}$ for some $p$ such that $\supp(x)\cap h\bmod{p^{2}}$ is finite; then shifting $x$ finitely many times will cause $x$ to omit an extra residue class mod $p^{2}$, removing $x$ from this set. Hence, to get an invariant set, define for each prime $p$ the set of $p$-persistent sequences in $X$
\[
 S_{p} = \{x\in X : \mbox{ for all } h\in\Z/p^{2}\Z, \mbox{ if }\supp(x) \ \cap \ h(p^{2})\mbox{ isn't empty, then it is infinite}\}.
\]
Observe that $S_{p}$ is an invariant subset of $X$. To see that $S_{p}$ is Borel, observe first that
\begin{equation}
 S_{p} = \bigcap_{h\in\Z/p^{2}\Z}\{x\in X: \supp(x) \ \cap \ h(p^{2})\mbox{ is empty}\}\amalg\{x\in X: \supp(x) \ \cap \ h(p^{2})\mbox{ is infinite}\}.
\label{eq: amalg}
\end{equation}
For each $h\in\Z/p^{2}\Z$, let $N_{h} = \{x\in X : \supp(x)\cap h(p^{2})\neq\emptyset\}$. Then 
\[
 N_{h} = \bigcup_{\ell = 1}^{\infty}\coprod_{\substack{w\in\mathcal{W}_{\ell}(X) \\ \supp(w)\cap h(p^{2})\neq\emptyset}}C_{w}
\]
where $C_{w}$ is the cylinder set defined by $w$ and $\mathcal{W}_{\ell}(X)$ is the set of all words of length $\ell$ that appear in the elements of $X$. Therefore, $N_{h}$ is a Borel subset of $X$ (in fact it's open). Hence, for each $h\in\Z/p^{2}\Z$, the first set comprising the disjoint union in the term corresponding to $h$ on the right side of eq. (\ref{eq: amalg}) is Borel, since it's the complement of $N_{h}$. At the same time, the second set in this disjoint union may be described as
\[
 \{x\in X: \supp(x) \cap h(p^{2})\mbox{ is infinite}\} = \bigcap_{m=0}^{\infty}\sigma^{-mp^{2}}N_{h}.
\]
This is immediate from the fact that $\sigma^{mp^{2}}x\in N_{h}$ whenever $h+Mp^{2}\in\supp(x)$ for some $M\geq m$. This intersection is clearly a Borel set, so $S_{p}$ is Borel as well. 

Now set
\begin{align*}
 X_{I} & = \{x\in X: \supp(x)\mbox{ omits exactly } i_{p}\mbox{ residue classes }\bmod{p^{2}} \mbox{ for all } p \\
   & \hspace{3cm}\mbox{ and } x\in S_{p} \mbox{ for all } p\}.
\end{align*}
Then this is an invariant subset of $X$. It's also Borel because
\[
  X_{I} = \bigcap_{p}\{x\in X : x\mbox{ omits exactly } i_{p}\mbox{ residue classes }\bmod{p^{2}}\}\cap S_{p}.
\]
We showed above that $S_{p}$ is Borel; at the same time we have for any $p$
\begin{align*}
 & \{x\in X : x\mbox{ omits exactly } i_{p}\mbox{ residue classes }\bmod{p^{2}}\} = \\
 & \hspace{2cm} \bigcup_{L\geq 1}\bigcap_{\ell\geq L}\coprod_{\substack{w\in W_{\ell}(X) \\ w\text{ omits exactly} \\ i_{p}\text{ residue classes }\bmod{p^{2}}}}C_{w}.
\end{align*}
This is clearly a Borel set, and it follows that $X_{I}$ is Borel.

We will denote $X_{1} = X_{(1,1,1,\dots)}$. Let $\mathcal{C}_{I} = \overline{X_{I}}$. Then $\mathcal{C}_{I}$ is a closed subshift of $X$. Also, for any $r\geq 1$ and $(i_{1},\dots,i_{r})\in\prod_{k=1}^{r}\{1,\dots,p_{k}^{2}-1\} $, define
\begin{align*}
 X_{(i_{1},\dots,i_{r})}(r) & = \{x\in X(r): \supp(x)\mbox{ omits exactly } i_{k}\mbox{ residue classes }\bmod{p_{k}^{2}} \mbox{ for } k=1,\dots, r \\
 & \hspace{3cm}\mbox{ and }x\in S_{p_{k}}\mbox{ for } k = 1,\dots,r\},
\end{align*}
and let $C_{(i_{1},\dots,i_{r})}(r) = \overline{X_{(i_{1},\dots,i_{r})}(r)}$. Note that for an infinite sequence $I = (i_{1},i_{2},\dots)$, we have $C_{I} = \cap_{r=1}^{\infty}C_{(i_{1},\dots,i_{r})}(r)$. By the Chinese remainder theorem we have the inequalities
\[
2^{m\prod_{k=1}^{r}\left(p_{k}^{2} - i_{k}\right)} \leq u^{I}_{mp_{1}^{2}p_{2}^{2}\cdots p_{r}^{2}} \leq \binom{b_{1}}{i_{1}}\binom{b_{2}}{i_2}\cdots\binom{p_{r}^{2}}{i_{r}}2^{m\prod_{k=1}^{r}\left(p_{k}^{2} - i_{k}\right)}, 
\]
where $u^{I}_{mp_{1}^{2}p_{2}^{2}\cdots p_{r}^{2}}$ is the number of distinct words of length $mp_{1}^{2}p_{2}^{2}\cdots p_{r}^{2}$ appearing in elements of $X_{I}(r)$. Therefore
\[
h_{\text{top}}(\mathcal{C}_{I}(r)) = \log 2\prod_{k=1}^{r}\left(1 - \displaystyle\frac{i_{k}}{p_{k}^{2}}\right)
\]
and so
\begin{align*}
h_{\text{top}}(\mathcal{C}_{I}) &= \lim_{r\to\infty}h_{\text{top}}(\mathcal{C}_{(i_{1},\dots,i_{r})}(r))\\
&= \lim_{r\to\infty}\log 2\prod_{k\leq r}\left(1 - \frac{i_{k}}{p_{k}^{2}}\right)\\
&=\log 2\prod_{p}\left(1 - \frac{i_{p}}{p^{2}}\right).
\end{align*}

\begin{lem}
Let $\nu$ be an ergodic measure on $X$. Then there exists a unique sequence $I\in\prod_{p}\{1,\dots,p^{2}-1\}$ such that $\nu(X_{I}) = 1$. In particular, any ergodic measure of maximal entropy on $X$ is concentrated on $X_{1}$.   
\label{lem: mme support}
\end{lem}
\begin{proof}
Although one might be tempted to reach the desired conclusion immediately from the fact that $X$ is the disjoint union of the $X_{I}$ over all sequences $I$, this is invalid since there are uncountably many such sequences, so we must argue differently. Let $\nu$ be an ergodic measure on $X$, and let $Y = \supp(\nu)$. For each $y\in Y$ and each $p$ let $i_{p}(y) = p^{2} - |\pi_{p^{2}}(\supp(y))|$. Then $i_{p}:Y\to\mathbb{N}$ is measurable, and we have for any $p$
\[
Y = \coprod_{i=1}^{p^{2} - 1}Y_{i}(p)
\]
where $Y_{i}(p) = \{y\in Y : i_{p}(y) = i\}$. Since $\nu$ is ergodic with $\nu(Y)=1$, and the $Y_{i}(p)$ are invariant and mutually disjoint, there exists a unique $i(p)\in\{1,\dots,p^{2} - 1\}$ such that $\nu(Y_{i(p)}) = 1$. As this is true for all $p$, and since a countable intersection of sets of full measure also has full measure, we have $\nu(Y\cap X_{I}) = \nu\left(\cap_{k}Y_{i(p)}\right) = 1$ where $I = (i(p))_{p}$, which proves the claim.
\end{proof}

Consider the map $\rho: X_{1}\to G$ given by sending a sequence $x$ to the sequence $g_{x}\in G$, where $(g_{x})_{p}$ is the unique residue class omitted by the support of $x\bmod{p^{2}}$. $\rho$ is measurable; indeed, if we fix finitely many primes $p_{1},\dots,p_{r}$ and classes $g_{i}\bmod{p_{i}^{2}}$ for $i=1,\dots, r$, then the inverse image under $\rho$ of the open set
\[
 \{g_{1}\}\times\{g_{2}\}\times\cdots\times\{g_{r}\}\times\prod_{p > p_{r}}\Z/p^{2}\Z
\]
is
\[
 X_{1}\cap\bigcap_{\ell\geq 1}\coprod_{\substack{w\in W_{\ell}(X) \\ \supp(w)\text{ omits } g_{i}\bmod{p_{i}^{2}} \\ \text{ for } i=1,\dots,r}}C_{w}
\]
which is clearly a Borel set.

Let 
\[
 Y = \{x\in X_{1} : g_{x}\in G_{1}^{c}\} = \rho^{-1}(G_{1}^{c}).
\]
$Y$ is Borel (since it's the inverse image of the closed set $G_{1}^{c}$ under the Borel map $\rho$) and recurrent. Indeed, any sequence in $Y$ must contain infinitely many 1's by virtue of belonging to $X_{1}$. Therefore, if one shifts such a sequence sufficiently many times so that there is a 1 in its first position, then there can't be any prime $p$ such that this shifted sequence omits $1\bmod{p^{2}}$ from its support.

Let $\sigma_{Y}: Y\to Y$ be the transformation induced by the shift $\sigma$ on $X$, as defined in Section \ref{sec:induced}. 
\begin{prop}
 $h(Y,\sigma_{Y}) = \log2$, where $h(Y,\sigma_{Y})$ is defined by eq. \ref{eq: entropy}.
 \label{prop: Y entropy}
\end{prop}
\begin{proof}
 Let $\nu$ be a $\sigma_{Y}$-ergodic Borel probability measure on $Y$. By fact (3), there exists a $\sigma$-ergodic Borel probability measure $\mu$ on $X$ such that $\nu = \mu_{Y}$. Since $\nu(Y) = 1$, we have $\mu(Y) > 0$, so by ergodicity, we have $\mu(X_{1}) = 1$ since $Y\subset X_{1}$ and $X_{1}$ is shift invariant. Therefore, $\rho_{*}\mu$ is a $T_{1}$-invariant probability measure on $G$, so by the unique ergodicity of $T_{1}$ it follows that $\rho_{*}\mu = m$, the mass one Haar measure on $G$. But we have $Y = \rho^{-1}(G_{1}^{c})$, so $\mu(Y) = \mu(\rho^{-1}(G_{1}^{c})) = m(G_{1}^{c})$. Observe that 
\[
 G_{1}^{c} = \prod_{p} (\Z/p^{2}\Z)\backslash\{1\}
\]
and therefore
\begin{align*}
 m(G_{1}^{c}) & = \lim_{K\to\infty}\frac{1}{p_{1}^{2}\cdots p_{K}^{2}}\prod_{k=1}^{K}(p_{k}^{2} - 1) \\
 & = \prod_{p}\left(1 - \frac{1}{p^{2}}\right).
\end{align*}
So we find
\begin{equation}
 \mu(Y) = \prod_{p}\left(1 - \frac{1}{p^{2}}\right) 
\label{eq: Y meas}
\end{equation}
for any ergodic measure $\mu$ on $X$ with $\mu(X_{1}) = 1$. Now, by Abramov's formula (fact (2) in Section \ref{sec:induced}) we have
\[
 h_{\mu_{Y}}(\sigma_{Y})\mu(Y) = h_{\mu}(\sigma) \leq h_{top}(X) = \log 2 \prod_{p}\left(1 - \frac{1}{p^{2}}\right)
\]
so by (\ref{eq: Y meas}), this yields $h_{\mu_{Y}}(\sigma_{Y})\leq \log 2$. On the other hand, there exists a measure of maximal entropy for $X$ since it is a subshift system, and this induces a measure of entropy $\log 2$ for $\sigma_{Y}$ by Abramov's formula once again; therefore $h(Y,\sigma_{Y}) = \log 2$ as claimed.
\end{proof}
Now, consider the product $G\times\Omega^{(2)}$ as before, with the skew-transformation $T$. We would like to induce $T$ to the subset $G_{1}^{c}\times\Omega^{(2)}$, but this is not strictly in keeping with our definitions since $G_{1}^{c}$ isn't a recurrent subset of $G$. However, observe by Poincar\'{e} recurrence that since the system $(G,T_{1},m)$ is ergodic and $m(G_{1}^{c})> 0$, there is a recurrent subset $R_{1}\subseteq G_{1}^{c}$ with $m(R_{1}) = m(G_{1}^{c})$, so we may unambiguously define the induced transformation $\overline{T} : R_{1}\times\Omega^{(2)}\to R_{1}\times\Omega^{(2)}$. Since $(G,T_{1})$ is uniquely ergodic, any $T$-invariant measure on $G\times \Omega^{(2)}$ must project to $m$ on the first factor. Therefore, for any $T$-invariant measure $\mu$ on $G\times\Omega^{(2)}$, the induced transformation $\overline{T} : G_{1}^{c}\times\Omega^{(2)}\to G_{1}^{c}\times\Omega^{(2)}$ is well-defined $\mu$-almost everyhwere. 

Since this introduces yet another system into our argument, we diagram the relations between the various systems of interest, where the diagram below is commutative:
\begin{align*}
 \sigma & : X \to X \\
 T & : G\times\Omega^{(2)} \to G \times \Omega^{(2)} & T(g,y) = (T_{1}g, \sigma^{\chi_{G_{1}^{c}}(g)}y) \\
 \iota & : G\times\Omega^{(2)} \to X & \mbox{ insertion map} \\
 \sigma_{Y} & : Y \to  Y & \mbox{ induced transformation of $\sigma$ to $Y \subset X$} \\
 \overline{T} & : G_{1}^{c}\times \Omega^{(2)}\to G_{1}^{c}\times\Omega^{(2)} & \mbox{ induced transformation of $T$} \\
 \overline{\iota} & : G_{1}^{c}\times\Omega^{(2)} \to X & \mbox{ restriction of $\iota$ to $G_{1}^{c}\times\Omega^{(2)}$} \\
 \overline{\iota}|_{W} & : W \to Y, & \mbox{ restriction of $\overline{\iota}$ to $W = \overline{\iota}^{-1}(Y)$}. 
\end{align*}
\begin{center}
\begin{tikzcd}
\overline{\iota}^{-1}(Y) = W \arrow{r}{\overline{T}_{W}} \arrow{d}{\overline{\iota}} 
& W \arrow{d}{\overline{\iota}} \arrow[hookrightarrow]{r}
& G\times\Omega^{(2)} \arrow{r}{T} \arrow{d}{\iota}
& G\times\Omega^{(2)} \arrow{d}{\iota}
\\
Y \arrow{r}{\sigma_{Y}}
& Y \arrow[hookrightarrow]{r}
& X \arrow{r}{\sigma} 
& X 
 
\end{tikzcd}
\end{center}

\begin{prop}
For any $T$-invariant measure $\mu$ on $G\times\Omega^{(2)}$, we have $\overline{T} = \overline{T_{1}}\times\sigma \ \mu$-a.e., where $\overline{T_{1}}:G_{1}^{c} \to G_{1}^{c}$ is induced from $T_{1}$. 
\label{prop: product}
\end{prop}
\begin{proof}
 Let $(g,y)\in G_{1}^{c}\times\Omega^{(2)}$. By definition, we have $\T(g,y) = T^{r'(g,y)}(g,y)$ where $r' : G_{1}^{c}\times\Omega^{(2)}\to\mathbb{N}$ is the return-time function. Now, if $r'(g,y) < \infty$ (this is the case for $\mu$-a.e. $(g,y)$ by the above), we have
 \begin{align*}
  r'(g,y) & = \min\{n\geq 1: T^{n}(g,y)\in G_{1}^{c}\times\Omega^{(2)}\} \\
  & = \min\{n\geq 1: T_{1}^{n}(g)\in G_{1}^{c}\} \\
  & = r_{1}(g)
 \end{align*}
where $r_{1} : G_{1}^{c}\to\mathbb{N}$ is the return-time function (which is well-defined for $m$-a.e. $g\in G_{1}^{c}$). Observe that for $g\in G_{1}^{c}$ with $r_{1}(g) < \infty$ and $1\leq n < r_{1}(g)$, we have $T_{1}^{n}g\in G_{1}$, and therefore $T^{n}(g,y) = (T_{1}^{n}g,\sigma y)$ by (\ref{eq: skew prod}). In particular, $T^{r_{1}(g)-1}(g,y) = (T_{1}^{r_{1}(g)-1}g,\sigma y)$ and $T_{1}^{r_{1}(g)-1}g\in G_{1}$, so we have 
\[
\T(g,y) = T^{r_{1}(g)}(g,y) = (T_{1}^{r_{1}(g)}g,\sigma y) = (\T_{1}\times\sigma)(g,y),
\]
which proves the claim.
\end{proof}

\begin{prop}
$ h(G_{1}^{c}\times\Omega^{(2)},\overline{T}) = \log 2$.
\label{prop: induced entropy}
\end{prop}
\begin{proof}
Let $\nu$ be a $\overline{T}$-invariant Borel probability measure on $G_{1}^{c}\times\Omega^{(2)}$. By Proposition \ref{prop: product}, we have $\overline{T} = \overline{T_{1}}\times\sigma \ \nu$-a.e. Hence if we let $\eta = (\pi_{2})_{*}\nu$, then $\eta$ is a $\sigma$-invariant measure on $\Omega^{(2)}$, and by the unique ergodicity of $\overline{T_{1}}, \ \nu$ projects to $\overline{m}$ on the first factor. Therefore, we have the inequality (cf. \cite{Dow} Fact 4.4.3)
\[
 h_{\nu}(\overline{T})\leq h_{\overline{m}}(\overline{T_{1}}) + h_{\eta}(\Omega^{(2)}) \leq \log 2,
\]
since $h_{\overline{m}}(\overline{T}_{1}) = 0$. At the same time, if we let $\nu = \overline{m}\times\omega_{(1/2,1/2)}$, where $\omega_{(1/2,1/2)}$ is the maximum entropy Bernoulli measure on $\Omega^{(2)}$, then
\[
 h_{\nu}(\overline{T}) = h_{\overline{m}}(\overline{T_{1}}) + h_{\omega_{(1/2,1/2)}}(\Omega^{(2)}) = \log 2,
\]
so the upper bound is achieved, proving the claim.

\end{proof}

We may define a map $\phi : Y\to G_{1}^{c}\times\Omega^{(2)}$ by $x\mapsto (g_{x}, \widehat{x})$, where $g_{x}\in G_{1}^{c}$ is the unique sequence of residue classes omitted by the support of $x$ modulo the $p^{2}$ (this belongs to $G_{1}^{c}$ by definition of $Y$) and $\widehat{x}\in\Omega^{(2)}$ is obtained by removing from $x$ the zeros along these residue classes. 
\begin{prop}
 For any $x\in Y$, we have $r_{1}(g_{x}) < \infty$.
\end{prop}
\begin{proof}
 If $r_{1}(g_{x}) = \infty$, then for any $n\geq 1$ there exists a prime $p$ such that $g_{\sigma^{n}x}\equiv 1\bmod{p^{2}}$. Hence, for any $n\geq 2$ there exists a $p$ such that $g_{x}\equiv n\bmod{p^{2}}$. But this clearly implies that $x = 00000\cdots$, which doesn't belong $Y$ since $Y\subset X_{1}$.   
\end{proof}
Therefore, the composite $\overline{T}\circ\phi: Y\to G_{1}^{c}\times\Omega^{(2)}$ is well-defined everywhere on $Y$, since $\overline{T}(g,y)$ is well-defined whenever $r_{1}(g) < \infty$.
\begin{prop}
$\phi\circ\sigma_{Y} = \overline{T}\circ\phi$ everywhere on $Y$. 
\label{prop: induced equivariance}
\end{prop}
\begin{proof}
 Let $x\in Y$; then $\phi(\sigma_{Y}x) = \phi(\sigma^{r(x)}x) = (g_{\sigma^{r(x)}x}, \widehat{\sigma^{r(x)}x})$. We have $g_{\sigma^{r(x)}x} = T_{1}^{r(x)}x = \overline{T_{1}}x$ since the map $X_{1}\to G, x\mapsto g_{x}$ intertwines $\sigma$ with $T_{1}$. We claim that $\widehat{\sigma^{r(x)}x} = \sigma(\widehat{x})$, which will prove our claim in view of Proposition \ref{prop: product}. 
 
 We may calculate for $m\geq 1$
 \[
  (\widehat{\sigma^{r(x)}x})_{m}  = (\sigma^{r(x)}x)_{\ell_{m}} = x_{\ell_{m} + r(x)},
 \]
where the sequence $\ell_{m}$ is defined for $m\geq 1$ by
\begin{gather*}
 \ell_{1} = \min\{\ell\geq 1: \ell \not\equiv g_{p} - r(x)\bmod{p^{2}} \ \forall p\} \\
 \ell_{m + 1}  = \min\{\ell > \ell_{m} : \ell\not\equiv g_{p} - r(x)\bmod{p^{2}} \ \forall p\}.
\end{gather*}
At the same time, we have
\[
 \widehat{x}_{m} = x_{\ell'_{m}},
\]
where $\ell'_{1} = 1$ and $\ell'_{m+1} = \min\{\ell > \ell'_{m} : \ell\not\equiv g_{p}\bmod{p^{2}} \ \forall p\}$. Observe that by definition of $Y$ and $r$ and the fact that $x\in Y$, for $1\leq n < r(x)$ there exists a prime $p_{n}$ such that $\sigma^{n}x$ omits the residue class $1\bmod{p_{n}^{2}}$. Therefore, $x$ itself omits the classes $n + 1\bmod{p_{n}^{2}}$ for $1\leq n < r(x)$. Thus, if $m > 1$ then $\ell'_{m} > r(x)$. It follows that $\ell_{m} + r(x) = \ell'_{m+1}$. Indeed, when $m = 1$ we have
\begin{align*}
 \ell'_{2} & = \min\{\ell > 1 : \ell\not\equiv g_{p}\bmod{p^{2}} \ \forall p\} \\
 & = \min\{\ell > r(x) : \ell\not\equiv g_{p}\bmod{p^{2}} \ \forall p\} \\
 & = r(x) + \min\{\ell\geq 1: \ell \not\equiv g_{p} - r(x)\bmod{p^{2}} \ \forall p\} \\
 & = r(x) + \ell_{1}
\end{align*}
and now for $m > 1$ it easily results from the inductive definitions of $\ell_{m}$ and $\ell'_{m}$. Therefore  
\[
 (\widehat{\sigma^{r(x)}x})_{m} = x_{\ell_{m} + r(x)} = x_{\ell'_{m+1}} = (\sigma(\widehat{x}))_{m}
\]
which proves the claim.

\end{proof}
Let $W = \overline{\iota}^{-1}(Y)$. Note that $\overline{\iota}|_{W}: W\to Y$ is bijective since $\overline{\iota}\circ\phi = Id_{Y}$. It follows that the restriction $\overline{\iota}|_{W} : W\to Y$ is a Borel isomorphism. Moreover, if $\phi(x)\in W$ (every element of $W$ may be realized this way for a unique $x\in Y$) then by Proposition \ref{prop: induced equivariance}
\[
\overline{\iota}\overline{T}\phi(x) = \overline{\iota}\phi\sigma_{Y}x  = \sigma_{Y}{x} = \sigma_{Y}\overline{\iota}\phi(x),                                                                                                                                \]                                                                                                                                                                                                                                                                                                                                                                                                                                                                                                                 
so $\overline{\iota}|_{W}$ is a measurable Borel isomorphism between the systems $(W,\overline{T}|_{W})$ and $(Y,\sigma_{Y})$:
\[
 \overline{\iota}|_{W} : (W,\overline{T}|_{W})\overset{\sim}{\to}(Y,\sigma_{Y}).
\]

In particular, $h(W,\overline{T}|_{W})= h(Y,\sigma_{Y}) = \log 2 = h(G_{1}^{c}\times\Omega^{(2)},\overline{T})$ where the second equality is Proposition \ref{prop: Y entropy} and the last equality is Proposition \ref{prop: induced entropy}. Since $(W, \overline{T}|_{W})$ is a subsystem of $(G_{1}^{c}\times\Omega^{(2)},\overline{T})$, it follows that: 
\begin{center}
 If $(G_{1}^{c}\times\Omega^{(2)},\overline{T})$ is intrinsically ergodic, then $(W, \overline{T}|_{W})$ is intrinsically ergodic.
\end{center}
Indeed, if we could find two distinct measures of maximal entropy for $\overline{T}|_{W}$, then these could be extended to distinct measures of maximal entropy for $\overline{T}$ on $G_{1}^{c}\times\Omega^{(2)}$. On the other hand, since $(W, \overline{T}|_{W})$ is Borel isomorphic to $(Y,\sigma_{Y})$, this implies the intrinsic ergodicity of the latter system. But $(Y,\sigma_{Y})$ is induced from $(X,\sigma)$. Thus, if $\nu$ is a $\sigma_{Y}$-ergodic measure of maximal entropy $\log2$ on $Y$, then by facts (1) and (3) in section \ref{sec:induced} there is a $\sigma$-ergodic measure $\mu$ on $X$ with $\nu = \mu_{Y}$, and then by Abramov's formula we have
\[
 h_{\mu}(X) = \mu(Y)h_{\nu}(\sigma_{Y}) = \mu(Y)\log 2 = \log2 \prod_{p}\left(1 - \frac{1}{p^{2}}\right) = h_{\text{top}}(X)
\]
where the second to last equality follows from (\ref{eq: Y meas}), since the fact that $\mu(Y) > 0$ implies $\mu(X_{1}) = 1$ since $\mu$ is ergodic and $Y\subset X_{1}$. This shows that the map $\mu\mapsto\mu_{Y}$ is a bijection between $\sigma$-ergodic measures of maximal entropy on $X$ (such a measure satisfies $\mu(Y) = 6/\pi^{2} > 0$ by (\ref{eq: Y meas}) and $\sigma_{Y}$-ergodic measures of maximal entropy on $Y$, and in particular intrinsic ergodicity for the two systems is equivalent. Combining this with the above observations, we see that
\begin{center}
  If $(G_{1}^{c}\times\Omega^{(2)},\overline{T})$ is intrinsically ergodic, then $(X,\sigma)$ is intrinsically ergodic.
\end{center}

\begin{prop}
 $(G_{1}^{c}\times\Omega^{(2)},\overline{T})$ is intrinsically ergodic.
\end{prop}
\begin{proof}
 By Proposition \ref{prop: product}, $\overline{T} = \overline{T_{1}}\times\sigma$ up to a set of measure zero for any invariant measure. Also, by the facts in Section \ref{sec:induced}, $\overline{T_{1}}$ is uniquely ergodic, with the zero entropy measure $\overline{m} = m_{G_{1}^{c}}$ as its unique invariant measure. Therefore, if $\mu$ is a $\overline{T}$-ergodic measure of maximal entropy $\log2$ on $G_{1}^{c}\times\Omega^{(2)}$, then its projection onto $G_{1}^{c}$ must be $(\pi_{1})_{*}\mu = \overline{m}$, while it projects to some $\sigma$-ergodic measure $\eta = (\pi_{2})_{*}\mu$ on $\Omega^{(2)}$. As in the proof of Proposition \ref{prop: induced entropy}, we have
 \[
  \log 2 = h_{\mu}(G_{1}^{c}\times\Omega^{(2)},\overline{T})\leq h_{\overline{m}}(G_{1}^{c},\overline{T_{1}}) + h_{\eta}(\Omega^{(2)},\sigma) = h_{\eta}(\Omega^{(2)},\sigma) \leq \log 2,
 \]
which shows that $h_{\eta}(\Omega^{(2)},\sigma) = \log 2$, and therefore $\eta = \omega_{(1/2,1/2)}$, the maximum entropy Bernoulli measure on $\Omega^{(2)}$. But the system $(G_{1}^{c},\overline{T_{1}},\overline{m})$ is ergodic with entropy zero, and any such system is \emph{disjoint} from the Bernoulli system $(\Omega^{(2)},\sigma,\omega_{(1/2,1/2)})$: the only invariant measure $\mu$ on the product system $(G_{1}^{c}\times\Omega^{(2)}, \overline{T}\times\sigma)$ with $(\pi_{1})_{*}\mu = \overline{m}$ and $(\pi_{2})_{*}\mu = \omega_{(1/2,1/2)}$ is the product measure $\overline{m}\times\omega_{(1/2,1/2)}$ (\cite{Fur} Thm. 1.2). As $\mu$ was chosen arbitrarily, it follows that this is the only measure of maximal entropy on $(G_{1}^{c}\times\Omega^{(2)},\overline{T})$.  

\end{proof}
It now follows from the facts in Section \ref{sec:induced} that $m\times\omega_{(1/2,1/2)}$ is the only measure of maximal entropy on the skew-product $(G\times\Omega^{(2)},T)$, which finally gives us the following.
\begin{thm}
$\iota_{*}(m\times\omega_{(1/2,1/2)})$ is the only measure of maximal entropy on $X$. 
\end{thm}

In fact, it's easy to see that the restriction of $\iota$ to the inverse image of $X_{1}$ in $G\times\Omega^{(2)}$ is bijective; since $X_{1}$ has full measure for $\eta = \iota_{*}(m\times\omega_{(1/2,1/2)})$, we see that $\iota^{-1}(X_{1})$ has full measure for $m\times\omega_{(1/2,1/2)}$, and this gives the following.

\begin{thm}
$\iota$ induces an isomorphism of measure-preserving systems
\[
\iota: (G\times\Omega^{(2)},T,m\times\omega_{(1/2,1/2)})\xrightarrow{\sim} (X,\sigma,\eta).
\]
\end{thm}


\section{Failure of the Gibbs property}\label{sec: gibbs}
An invariant probability measure $\mu$ supported on a closed subshift $X$ of the full 2-shift $\Omega^{(2)}$ is said to be a Gibbs measure if there exists a constant $c > 0$ such that for any $n\geq 1$ and any word $w$ of length $n$ appearing in $X$, we have the inequality 
\[
 \mu(C_{w}) \geq c\cdot e^{-n\cdot h_{\text{top}}(X)}
\]
where $C_{w} \subset X$ is the cylinder set defined by $w$. In other words, the quantity $e^{|w|\cdot h_{\text{top}}(X)}\mu(C_{w})$ is uniformly bounded below as $w$ varies over all words in $X$. The importance of this property is the following:
\begin{prop}[\cite{Wei} Lemma 2]
Suppose $X$ is a closed subshift of $\Omega^{(2)}$ and $\mu$ is an ergodic Gibbs measure supported on $X$ such that $h_{\mu}(X) = h_{\text{top}}(X)$. Then $\mu$ is the only measure of maximal entropy on $X$. 
\end{prop}

The unique measures of maximal entropy on subshifts of finite type and sofic systems are Gibbs measures. Our aim in this section is to show that the unique measure of maximal entropy on the squarefree flow $X$ does not possess the Gibbs property, in contrast to many well-known classes of intrinsically ergodic systems.

We have shown that the measure $\eta = \iota_{*}(m\times\omega_{(1/2,1/2)})$ is the only measure of maximal entropy on the squarefree flow $X$. This allows us to give an explicit formula for $\eta(w):= \eta(C_{w})$ for any word $w$ in $X$, in the following way. Let $G(w) = \{g\in G:  \forall p, \ g_{p}\not\in\supp(w)\bmod{p^{2}}\}$. Observe that if $C_{w}$ is the cylinder defined by $w$, then
\begin{align*}
& \iota^{-1}(C_{w}) = \{(g,y)\in G\times\Omega^{(2)} : g\in G(w) \text{ and }y_{m - \alpha_{g}(m)} = w_{m}\text{ whenever } m\leq n \\ & \hspace{6cm} \text{ and }\delta_{g}(m) = 0\}
\end{align*}
where $\delta_{g}(m) = 0$ if and only if there does not exist any prime $p$ such that $m\equiv g_{p}\bmod{p^{2}}$. Hence, by Fubini's theorem and the definition of the Bernoulli measure we have
\[
\eta(w) = \iota_{*}(m\times\omega_{(1/2,1/2)})(C_{w}) = \int_{G(w)}2^{-|\{m\leq n: \ \delta_{g}(m) = 0\}|}dm(g).
\]
By definiton, we have for $g\in G(w)$ that $\supp(w)\subseteq\{m\leq n: \ \delta_{g}(m) = 0\}$. Therefore, we have $2^{-|\{m\leq n: \ \delta_{g}(m) = 0\}|}\leq 2^{-|\supp(w)|}$ for any $g\in G(w)$, so we get $\eta(w)\leq 2^{-|\supp(w)|}m(G(w))$. For $r\geq 1$, define $G_{r}(w) = \{g\in G_{r} : g_{i}\not\in\supp(w)\bmod p_{i}^{2} \text{ for } i=1,\dots,r\}$. By expressing the Haar measure on $G$ as the limit of the counting measures on the $G_{r}$ as $r\to\infty$, we find
\[
\eta(w) \leq 2^{-|\supp(w)|}\lim_{r\to\infty}\frac{|G_{r}(w)|}{p_{1}^{2}\cdots p_{r}^{2}}.
\]
For any $r\geq 1$ we have
\[
\frac{|G_{r}(w)|}{p_{1}^{2}\cdots p_{r}^{2}} = \prod_{i=1}^{r}\left(1 - \frac{u(w,i)}{p_{i}^{2}}\right)
\]
where $u(w,i)$ is the number of residue classes $\bmod p_{i}^{2}$ defined by $\supp(w)$. Observe that if $p_{i}^{2} > n$ then $u(w,i) = |\supp(w)|$. Since every term in the product is less than 1 we get
\begin{equation}
\frac{|G_{r}(w)|}{p_{1}^{2}\cdots p_{r}^{2}} \leq \prod_{\substack{1\leq i\leq r \\ p_{i} > \sqrt{n}}}\left(1 - \frac{|\supp(w)|}{p_{i}^{2}}\right),
\end{equation}
so letting $r\to\infty$, we find
\begin{equation}
\eta(w) \leq 2^{-|\supp(w)|}\prod_{\substack{p \\ p > \sqrt{n}}}\left(1 - \frac{|\supp(w)|}{p^{2}}\right).
\label{eq: ineq}
\end{equation}
For each $n\geq 1$, let $w^{(n)}$ be the word of length $n$ appearing in $X$ defined by $w^{(n)}_{k} = \mu^{2}(k)$ for $1\leq k\leq n$. Clearly, $|\supp(w^{(n)})| = Q(n)$, the number of squarefree integers less than or equal to $n$. Therefore, the above inequality yields
\[
 \eta(w^{(n)}) \leq 2^{-Q(n)}\prod_{p>\sqrt{n}}\left(1 - \frac{Q(n)}{p^{2}}\right).
\]
We proceed to estimate the product; call it $P(n)$. Taking logs and using the Taylor series for $\log(1-x)$ gives
\begin{align*}
 \log P(n) & = \sum_{p>\sqrt{n}}\log\left(1 - \frac{Q(n)}{p^{2}}\right) \\
 & \sim -Q(n)\sum_{p >\sqrt{n}}\frac{1}{p^{2}}.
\end{align*}
By Riemann-Stieltjes integration and the prime number theorem we therefore get
\[
 \log P(n)\sim -2\frac{Q(n)}{\sqrt{n}\log n}
\]
as $n\to\infty$, so we find from the above
\[
 \eta(w^{(n)}) \ll 2^{-Q(n)}\exp\left(-2\frac{Q(n)}{\sqrt{n}\log n}\right) 
\]
as $n\to\infty$, with an absolute implied constant.

Now, write $Q(n)$ as
\[
 Q(n) = \frac{6}{\pi^{2}}n + R(n).
\]
The error term $R(n)$ satisfies (\cite{P})
\begin{equation}
 R(n) = o\left(\frac{\sqrt{n}}{(\log N)^{A}}\right) 
\label{eq: error}
\end{equation}
for any $A > 0$. Since $h_{\text{top}}(X) = (6/\pi^{2})\log 2$ we have 
\begin{align*}
 e^{n\cdot h_{\text{top}}(X)}\eta(w^{(n)}) & = 2^{6n/\pi^{2}}\eta(w^{(n)}) \\
 & \ll 2^{6n/\pi^{2}}2^{-Q(n)}\exp\left(-2\frac{Q(n)}{\sqrt{n}\log n}\right) \\
 & \ll \exp\left(-2\frac{Q(n)}{\sqrt{n}\log n} - R(n)\log2\right)
\end{align*}
and using the above expression for $Q(n)$ and the error estimate (\ref{eq: error})  we see 
\[
 e^{n\cdot h_{\text{top}}(X)}\eta(w^{(n)}) \ll \exp\left(-\frac{12}{\pi^{2}}\frac{\sqrt{n}}{\log n} + o\left(\frac{\sqrt{n}}{\log n}\right)\right).
\]
Since the exponent tends to $-\infty$ as $n\to\infty$, we see that 
\[
 e^{n\cdot h_{\text{top}}(X)}\eta(w^{(n)}) \to 0 \text{ as } n\to\infty.
\]
Since the Gibbs property precisely states that the quantity on the left is uniformly bounded below for all $n$ and $w$ with length $n$, we have shown the following.
\begin{prop}
$\eta$ is not a Gibbs measure.
\end{prop}    

In a similar vein, the above arguments show that any $\mathscr{B}$-free shift as mentioned in the introduction and described in \cite{ELR} carries a unique measure of maximal entropy (this is also proven in \cite{KLW}), and one can ask whether the Gibbs property holds for this measure. Given that the argument in this section used specific properties of the squares of primes, it's unclear whether or not this should be the case in general.


\begin{thebibliography}{MarNew}

\bibitem[Ab]{Ab}
\newblock L.M. Abramov, \emph{Entropy of induced automorphisms}, Dokl. Akad. Nauk. SSSR \textbf{128} (1959), 647 -- 650.

\bibitem[CS]{CS}
\newblock F. Cellarosi and Y. Sinai, \emph{Ergodic properties of square-free numbers}, J. Eur. Math. Soc., \textbf{15} 4 (2013), 1343--1374. 

\bibitem[CT]{CT}
\newblock V. Climenhaga and D. Thompson, \emph{Intrinsic ergodicity beyond specification: $\beta$-shifts, $S$-gap shifts, and their factors}, Israel J. Math \textbf{192} (2012), 785 -- 817.

\bibitem[Dow]{Dow}
\newblock T. Downarowicz, \emph{Entropy in Dynamical Systems}, Cambridge University Press (2011).

\bibitem[Ei]{Ei}
\newblock M. Einsiedler, E. Lindenstrauss, P. Michel, and A. Venkatesh, \emph{The distribution of closed geodesics on the modular surface, and Duke's theorem}, L'Enseignement Mathematique \textbf{2} 58 (2012), 249--313. 

\bibitem[ELR]{ELR}
\newblock E. El Abdalaoui, M. Lemanczyk, and T. De La Rue, \emph{A dynamical point of view on the set of $\mathscr{B}$-free integers}, preprint, http://arxiv.org/abs/1311.3752 (2013). 

\bibitem[Fur]{Fur}
\newblock H. Furstenberg, \emph{Disjointness in ergodic theory, minimal sets, and a problem in diophantine approximation}, Math. Syst. Theory, \textbf{1} 1 (1967), 1--49. 

\bibitem[Hall]{Hall}
\newblock R.R. Hall, \emph{The distribution of square-free numbers}, J. Reine Angew. Math \textbf{394} (1989), 107 -- 117.

\bibitem[Hay]{Hay}
\newblock N. Haydn, \emph{Phase transitions in one-dimensional subshifts}, preprint, http://arxiv.org/abs/1111.6227 (2011).

\bibitem[IK]{IK}
\newblock H. Iwaniec and E. Kowalski, \emph{Analytic number theory}, AMS Colloquium Publications \textbf{53} (2004).
 
\bibitem[Kec]{Kec}
\newblock A. Kechris, \emph{Classical Descriptive Set Theory}, Springer-Verlag (1995).

\bibitem[KLW]{KLW}
\newblock J. Kulaga-Przymus, M. Lemanczyk, B. Weiss, \emph{On invariant measures for $\mathcal{B}$-free systems}, preprint, arxiv.org/abs/1406.3745 (2014).


\bibitem[MarNew]{MarNew}
\newblock B. Marcus, S. Newhouse, \emph{Measures of maximal entropy for a class of skew products}, Springer Lecture Notes in Math, \textbf{729} (1978), 105 -- 125.

\bibitem[P]{P}
\newblock J. Pintz, \emph{On the distribution of square-free numbers}, J. London Math Soc \textbf{28} (1983), 401 -- 405.

\bibitem[Sar]{Sar}
\newblock P. Sarnak, \emph{Three lectures on the M\"obius function, randomness, and dynamics}, online notes, http://publications.ias.edu/sarnak/paper/512 (2011).

\bibitem[Tsa]{Tsa}
\newblock K.M. Tsang, \emph{The distribution of $r$-tuples of squarefree numbers}, Mathematika \textbf{32} (1985), 265--275.

\bibitem[Th]{Th}
\newblock M. Thaler, \emph{Transformations on $[0,1]$ with infinite invariant measures}, Israel J. Math, \textbf{46} (1983), 67--96.

\bibitem[Wal]{Wal}
\newblock P. Walters, \emph{An Introduction to Ergodic Theory}, Springer (1982).

\bibitem[Wei]{Wei}
\newblock B. Weiss, \emph{Subshifts of finite type and sofic systems},  Monats. Math. \textbf{77} (1973), 462 -- 474.

\bibitem[Zwei]{Zwei}
\newblock R. Zweim\"{u}ller, \emph{Invariant measures for general(ized) induced transformations}, Proc. Amer. Math. Soc. \textbf{133} (2005), 2283--2295. 



 
\end{thebibliography}
\end{document}